\documentclass[12pt]{amsart}

\usepackage{amsthm,amssymb,amstext,amscd,amsfonts,amsbsy,amsxtra,latexsym,amsmath}
\usepackage{fullpage}
\usepackage[english]{babel}
\usepackage[latin1]{inputenc}
\usepackage{verbatim} 
\usepackage{graphicx} 
\usepackage[all]{xy} 
\numberwithin{figure}{section}

\newcommand{\field}[1]{\mathbb{#1}} 
\newcommand{\N}{\field{N}}  
 
\newcommand{\R}{\field{R}}
\newcommand{\C}{\field{C}}

\newcommand{\HH}{\mathbb{H}}

\newcommand{\im}{\text{Im}}
\newcommand{\re}{\text{Re}}
\renewcommand{\H}{\mathbb{H}}

\newcommand{\CC}{\mathcal{C}}
\newcommand{\cD}{\mathcal{D}} 
\newcommand{\CM}{\mathcal{M}}

\newcommand{\non}{\nonumber}

\newcommand{\sech}{\operatorname{sech}}
\newcommand{\csch}{\operatorname{csch}}
\numberwithin{equation}{section}
\newtheorem{theorem}{\textbf{Theorem}}
\numberwithin{theorem}{section}
\newtheorem{corollary}{\textbf{Corollary}}
\numberwithin{corollary}{section}
\newtheorem{lemma}[theorem]{\textbf{Lemma}}
\newtheorem{proposition}[theorem]{\textbf{Proposition}}

\newtheorem*{remark}{Remark}

\renewenvironment{proof}[1][Proof]{\begin{trivlist}
\item[\hskip \labelsep {\bfseries #1:}]}{\qed\end{trivlist}}

\newcommand{\bea}{\begin{eqnarray}} 
\newcommand{\eea}{\end{eqnarray}} 
\newcommand{\be}{\begin{equation}} 
\newcommand{\ee}{\end{equation}} 
\newcommand{\benn}{\begin{equation*}} 
\newcommand{\eenn}{\end{equation*}}

\begin{document}

\title[]{The asymptotic profile of $\chi_y-$genera of Hilbert schemes of points
  on K3 surfaces}
\begin{abstract}
The Hodge numbers of the Hilbert schemes of points on algebraic surfaces are given by G\"ottsche's formula, which expresses the generating functions of the Hodge numbers in terms of theta and eta functions. We specialize in this paper to generating functions of the $\chi_y$-genera of Hilbert schemes of $n$ points on K3 surfaces. We determine asymptotic values of the coefficients of the $\chi_y$-genus for $n\to \infty$ as well as their asymptotic profile. 
\end{abstract} 

\author{Jan Manschot$^\dagger$}
\address{$^\dagger$ School of Mathematics \\ Trinity College \\ College Green \\Dublin 2\\
Ireland}
\email{manschot@maths.tcd.ie}

\author{Jose Miguel Zapata Rolon$^*$}
\address{$^*$ Mathematical Institute\\University of
Cologne\\ Weyertal 86-90 \\ 50931 Cologne \\Germany}
\email{rzapata@math.uni-koeln.de}

\maketitle
\section{Introduction and Results}

The Hilbert scheme $S^{[n]}$ of $n$ points on a complex projective surface $S$ heuristically parametrizes collections of $n$ points on the surface $S$. The geometry of such Hilbert schemes is well studied.\footnote{See for example \cite{Gottsche:2003} and \cite{Nakajima:1999} for two expository texts.} In this paper we will be mainly considering the case when $S$ is a K3 surface. K3 surfaces are smooth, compact and simply connected surfaces with trivial canonical bundle. K3 surfaces are hyper-K\"ahler manifolds and exhibit a wealth of other special properties. The Hilbert schemes $\mathrm{K3}^{[n]}$ are also hyper-K\"ahler manifolds, and their topological invariants determine many other interesting invariants for mathematical objects associated to $\mathrm{K3}$: Gromov-Witten invariants \cite{Katz:1999xq, Kawai:2000px}, stable pair invariants \cite{Pandharipande:2014}, rank $r$ sheaves of pure complex dimension 2 \cite{Vafa:1994tf}. K3 surfaces are important in Calabi-Yau compactifications of 10-dimensional string theory to 4 and 6 dimensions. From this perspective, the Hodge numbers $h^{p,r}(S^{[n]})$ of the cohomology give information about the number of (supersymmetric) quantum states in the lower dimensional physical theories. See for example \cite{Diaconescu:2007bf}. 

The Hodge numbers $h^{p,r}\!(S^{[n]})$ of the Hilbert schemes of $n$ points on an algebraic surface $S$  are famously given by G\"ottsche's formula \cite{Gottsche:1990}. G\"ottsche's formula expresses the generating function as an infinite product, and is in fact a simple product of Jacobi theta and Dedekind eta functions. The asymptotic growth of the Euler number $\chi(S^{[n]})$ for $n\to \infty$ has been known for a long time. It was found using the Rademacher circle method (see for example \cite{Apostol:1976, HR2}) and is of interest for conformal field theory and string theory. Recently, methods are also developed to derive the asymptotic behavior of the Betti numbers which gives much more refined information about the cohomology of the Hilbert schemes \cite{Bringmann:2013, Bringmann:2013yta}. Closely related is the work by Hausel and Rodriguez-Villegas \cite{Hausel:2013}, who have determined the asymptotic profiles of Betti numbers of a class of hyper-K\"ahler manifolds. The majority of those hyper-K\"ahler manifolds in \cite{Hausel:2013} appear as moduli spaces of families of mathematical objects. 

In the present paper we extend the techniques developed in \cite{Bringmann:2013, Bringmann:2013yta}, to determine the asymptotic behavior of the $\chi_y$ genus of $\operatorname{K3}^{[n]}$. To explain the setup and results in more detail, let 
\be
e(\CM;x,y):=\sum_{p,r=0}^{\dim_\mathbb{C}\CM} h^{p,r}\!(\CM)\,x^py^r, \non
\ee
be the Hodge polynomial of a smooth complex manifold $\CM$ with
$h^{p,r}(\CM)=\dim H^{p,r}(\CM)$ the dimension of the Dolbeault
cohomology group $H^{p,r}(\CM)$. If $\CM$ satisfies Poincar\'e duality, $e(\CM;x,y)$ is a palindromic polynomial in the two variables $x$ and $y$. 
The polynomial specializes to several other well-known characteristic polynomials. For $x=y$, one obtains
the Poincar\'e polynomial, i.e. the generating function of Betti
numbers $b_k(\CM)=\sum_{p+r=k} h^{p,r}(\CM)$. For $x=-1$, $e(\CM;x,y)$ specializes to the
$\chi_y$-genus of $\CM$: $\chi_y(\CM)=\sum_{p,r}
h^{p,r}(\CM)\,(-1)^py^r=\sum_{r}\chi^r(\CM) y^r$. The number
$\chi^r(\CM)$ is the index of the Dolbeault complex of forms with non-holomorphic degree $r$.  
Finally for $x=y=-1$, $e(\CM;x,y)$ equals the Euler number $\chi(\CM)$. 

The famous formula by G\"ottsche \cite[Conjecture 3.1]{Gottsche:1990} expresses the generating
function of Hodge polynomials of the Hilbert schemes as an infinite
product formula
 \be
\label{eq:genHodge}
 \sum_{n=0}^\infty e(S^{[n]};x,y)\, x^{-n}y^{-n}\,q^n=\prod_{n=1}^\infty \frac{\prod_{p+r=\mathrm{odd}}(1+x^{p-1}y^{r-1}q^n)^{h^{p,r}\!(S)}}{\prod_{p+r=\mathrm{even}}(1-x^{p-1}y^{r-1}q^n)^{h^{p,r}\!(S)}}.
 \ee
To specialize Eq. (\ref{eq:genHodge}) to a K3 surface, note that the nonvanishing Hodge numbers of a K3 surface are given by
$h^{0,0}(\mathrm{K3})=h^{2,0}(\mathrm{K3})=h^{0,2}(\mathrm{K3})=h^{2,2}(\mathrm{K3})=1$, and $h^{1,1}(\mathrm{K3})=20$. 

In the physical context of Calabi-Yau compactifications to 4 dimensions, the exponents of $x$ and $y$ label representations of the SU(2) rotation and U(1) R-symmetry group \cite{Diaconescu:2007bf}. Specializing (\ref{eq:genHodge}) by $x\to -1$, and $y\to -\zeta$, one obtains a generating function of the Laurent polynomials $\zeta^{-n}\,\chi_{-\zeta}(\mathrm{K3}^{[n]})$. The exponent of $\zeta$ then labels representations of a diagonal subgroup  $\subset\operatorname{SU(2)}\times \operatorname{U(1)}$.

To obtain our results we use and develop techniques from analytic number theory. See for example the closely related papers \cite{Bringmann:2012bm, Bringmann:2013, Bringmann:2012, Bringmann:2013yta, Dousse:2014, HR2, Kim:2014, Mainka:2014, Rolon:2014}. We expect that the techniques in this paper might in turn motivate and be relevant for questions in analytic number theory. For example, a combinatorial interpretation of the coefficients of (\ref{eq:genHodge}) (and its specialization of $(x,y)$ to $(-1,-\zeta)$)  in terms of colored partitions is still missing. We continue by expressing the generating function in terms of modular forms. Recall that the Jacobi theta function $\vartheta(w;\tau)$ is defined for $w \in \C$ and $\tau \in \HH$
\be
\label{eq:theta}
\vartheta(w;\tau) := i \zeta^{\frac{1}{2}} q^{\frac{1}{8}}\prod_{n=1}^{\infty}\left(1-q^n\right)\left(1-\zeta q^n\right)\left(1-\zeta^{-1}q^{n-1}\right)=i\zeta^\frac{1}{2}q^\frac{1}{8}\,(q;q)_\infty\,(q\zeta;q)_\infty\,(\zeta^{-1};q)_\infty, \non
\ee
where $\zeta := e^{2 \pi i w}$, $q:= e^{2\pi i \tau}$ and $(a;q)_\infty=\prod_{n=0}^\infty(1-aq^n)$ is the $q$-Pochhammer symbol. Recall also that the Dedekind $\eta$-function is defined as
\be
\label{eq:eta} \non
 \eta(\tau) := q^{\frac{1}{24}}\prod_{n=1}^{\infty}(1-q^n)=q^\frac{1}{24}(q;q)_\infty.
\ee
Then specialization of equation (\ref{eq:genHodge}) to $x=-1$ and $y=-e^{2\pi iw}=-\zeta$ gives the generating function of $\chi_{-\zeta}$-genera of $\mathrm{K3}^{[n]}$
\[
f_k(w; \tau):= \frac{g(w;\tau)^2}{\eta(\tau)^{k}}, 
\]
with $k=24$ and
\[
g(w; \tau):=i\,\frac{\left(\zeta^{\frac12}-\zeta^{-\frac12}\right) \eta(\tau)^{3}}{\vartheta(w; \tau)}.
\] 

We define the coefficients $a_{m,k}(n)$ of $f_k(w,\tau)$, $k\geq 1$, as follows $$f_k(w,\tau) := \sum_{m,n} a_{m,k}(n) \zeta^m  q^{n-\frac{k}{24}}.$$ We note that the coefficients $\chi^r(\mathrm[K3]^{[n]})$ are given by $(-1)^r a_{r,24}(n)$.  Due to the symmetry $w\leftrightarrow -w$ of $f_k(w;\tau)$ it is easy to deduce that $\chi_{-\zeta}(\mathrm{K3}^{[n]})$ is a palindromic polynomial of degree $2n$ with positive coefficients. We note that $\sum_{m\in \mathbb{Z}} a_{m,k}(n)=p_k(n)$, with $p_k(n)$ the number of partitions of $n$ in $k$ colors. 

Our first result is obtained using the approach of Wright \cite{Wright:1971} also used in \cite{Bringmann:2012bm, Bringmann:2012, Bringmann:2013yta,Mainka:2014, Rolon:2014}. 
\begin{theorem}\label{th1}  
Let $N \in \N$ and $I_\ell(x)$ denote the $I$-Bessel function defined in equation (\ref{eq:IBessel}). For fixed $m$ we have, as $ n \rightarrow \infty$
\begin{eqnarray}
a_{m,k}(n)&=&(2\pi)^{-\frac{k}{2}}\sum_{\ell=1}^N d_{m,k}(\ell)\,n^{-\frac{2+2\ell+k}{4}} \left( \pi \sqrt{\frac{ k}{6}}\right)^{1+\ell+\frac{k}{2}} \non\\
&&\times I_{-1-\ell-\frac{k}{2}}\!\left(\pi \sqrt{\frac{2k n}{3} } \right)+O\!\left(n^{-1-\frac{N}{2}-\frac{k}{4}}e^{\pi \sqrt{\frac{2kn}{3}}}\right), \nonumber
\end{eqnarray}
\noindent where the $d_{m,k}(\ell)$ are defined in equation (\ref{eq:dmkz}). \end{theorem}

Since $d_{m,k}(\ell)$ are independent of $m$
 for $\ell=1,2$ and $d_{m,k}(3)-d_{r,k}(3)=\frac{1}{60}(r^2-m^2)$, we deduce:
\begin{corollary}\label{cor1}
The difference $a_{m,k}(n)-a_{r,k}(n)$, as $n \rightarrow \infty$, is given by
\begin{equation} \non
a_{m,k}(n)-a_{r,k}(n)=\frac{4}{15} \pi^3 (r^2-m^2)\,(8n)^{-\frac{9+k}{4}} \left( \frac{k}{3} \right)^{\frac{k+7}{4}}\,e^{\pi \sqrt{\frac{2kn}{3}}}+O\!\left(n^{-3-\frac{k}{4}}e^{\pi \sqrt{\frac{2kn}{3}}}\right).
\end{equation}
\end{corollary}
\begin{remark}
The leading asymptotic behavior of a similar difference of coefficients was determined in \cite[Corollary 1.2]{Bringmann:2013yta} for the function $g(w;\tau)/\eta(\tau)^k$.
We observe that the two asymptotic behaviors only differ by the factor $\frac{4}{15}$. Note however that for large $m$ the coefficients $d_{m,k}(\ell)$ grow much faster in the case studied here compared to \cite{Bringmann:2013yta}.
\end{remark}

Our second main result concerns the profile of the coefficients $a_{m,k}(n)$ for $|m|\leq  \sqrt{\frac{n}{6k}}\,\frac{\log n}{\pi}$. To this end we define 
\begin{equation}\label{eqn:profile}
 P(m,\beta):=\frac{d^2}{dm^2}\left(\frac{m}{2}\coth\!\left(\frac{\beta m}{2}\right)\right)=\frac{\beta}{4}\operatorname{csch}^2\!\left(\frac{\beta m}{2}\right) \left(\beta m\operatorname{coth}\!\left(\frac{\beta m}{2}\right)-2\right),
\end{equation} 
with $\csch(x)=1/\sinh(x)$. The function $P(m,\beta)$ satisfies $P(0,\beta)=\frac{\beta}{6}$, $\int_{-\infty}^\infty P(m,\beta)dm=1$ and has variance $\int_{-\infty}^\infty m^2 P(m,\beta) dm=\frac{2\pi^2}{3\beta^2}$. Using the Taylor expansion of $P(m,\beta)$ in $m$, we obtain the limiting shape of the ratio $a_{m,k}(n)/p_k(n)$ for large $n$. This is given by:
\begin{theorem}\label{th2} Let $p_k(n)$ be the number of partitions of $n$ in $k$ colors. For $m$ as above we have, as $n \rightarrow \infty$
$$
\frac{a_{m,k}(n)}{p_k(n)}= 
P(m,\beta_k)\left(1 + O\!\left(\beta_k^{\frac{1}{2}}|m|^{\frac{1}{3}}\right)\right),
$$
where $\beta_k=\pi\sqrt{\frac{k}{6n}}$.
\end{theorem}

It is an interesting open question to which distribution the probability density function $P(m,\beta)$ corresponds. Probability distributions occurred earlier for coefficients of inverse theta functions and for the cohomology of hyper-K\"ahler manifolds. For example the profile for the function $g(w;\tau)/\eta(\tau)^k$ was conjectured by Dyson \cite{Dyson:1989} (and recently proven by Bringmann and Dousse \cite[Theorem 1.3]{Bringmann:2013}; see also \cite[equation (2.13)]{Russo:1994ev}),  to be equal to $P_{\mathrm{log}}(m,\beta)=\frac{\beta}{4} \sech^2\!\left( \textstyle\frac{1}{2}\beta m \right)=\frac{\beta}{4}
\cosh^{-2}\!\left( \textstyle\frac{1}{2}\beta m \right)$, which coincides with the probability density function of the logistic distribution with mean 0 and variance $\frac{\pi^2}{3\beta^2}$. Similarly, the profiles of Betti numbers of hyper-K\"ahler manifolds found by Hausel and Rodriguez-Villegas allow often an interpretation as probability distributions \cite{Hausel:2013}. For example, the profile of the Betti numbers of Hilbert schemes on $\mathbb{C}^2$ corresponds to the Gumbel distribution \cite{Hausel:2013}. In a similar spirit, a Gaussian distribution is found for DT-invariants of $\mathbb{C}^3$ \cite{Morrison:2013}.  

The paper is organized as follows: In section \ref{sec:th1} we prove Theorem \ref{th1} by using Wright's method as in \cite{Bringmann:2013yta} and in section \ref{sec:th2} we adopt the method developed in \cite{Bringmann:2013} to our generating function of $\chi_y$-genera to prove Theorem \ref{th2}.
\section*{Acknowledgments}
We would like to thank Kathrin Bringmann for useful discussions and collaboration on closely related subjects. The first author also thanks Tam\'as Hausel for discussions. Part of the reported research was carried out while the first author was affiliated with the Camille  Jordan Institute of the University of Lyon. The research of the second author was supported by the DFG-Graduiertenkolleg ``Globale Strukturen in Geometrie und Analysis''. The second author would like to thank the Camille Jordan Institute and Trinity College Dublin for the hospitality during his stay and Michael Mertens for many useful discussions. We also thank the anonymous referees for their suggestions.

\section{Proof of theorem \ref{th1}}
\label{sec:th1}
\subsection{The main term}
We start by recalling the transformation properties under $\tau \to -1/\tau$ of the Jacobi theta and the Dedekind eta function.
\begin{lemma}\label{modular}
We have
\begin{eqnarray}
&&\eta\!\left( -\frac{1}{\tau} \right) = \left( - i \tau \right)^{\frac{1}{2}}\eta(\tau), \non \\
&&\vartheta\!\left(\frac{w}{\tau};-\frac{1}{\tau}\right)= - i \left(-i\tau\right)^{\frac{1}{2}} e^{\pi i w^2/\tau}\vartheta(w;\tau). \non
\end{eqnarray}
\end{lemma}
\noindent To prove Theorem \ref{th1} we investigate the main term of $g(w;\tau)$ by using the transformation rules of $\eta(\tau)$ and $\vartheta(w;\tau)$. Recall that $q:=e^{2\pi i \tau}$ and $\zeta:=e^{2\pi i w}$. We define furthermore $z:=-2\pi i \tau$ such that $q=e^{-z}$. Then we have for $g\!\left(w; \frac{iz}{2\pi}\right)^2$ as $z\to 0$:
\begin{lemma}\label{approx}
For $0<\mathrm{Re}(z) \ll 1$, $ 0 \leq \mathrm{Re}(w) \leq 1$ we have
$$g\!\left(w; \frac{iz}{2\pi}\right)^2 =  \frac{\sin(\pi w)^2\exp\!\left(\frac{4\pi^2 w^2}{z}\right)}{\left(\frac{z}{2\pi}\right)^2\sinh\!\left(\frac{2\pi^2 w}{z}\right)^2}
\left(1+O\!\left(e^{-4\pi^2\mathrm{Re}\left(\frac1z\right)(1-w)}\right)\right).$$
\end{lemma}
\begin{proof}
Using Lemma \ref{modular}, we obtain
\[
g\!\left(w; \frac{iz}{2\pi}\right)^2=-\frac{\left(\zeta^{\frac12}-\zeta^{-\frac12}\right)^2\eta\!\left(\frac{iz}{2\pi}\right)^{6}}{\vartheta\!\left(w; \frac{iz}{2\pi}\right)^2} 
=\frac{\left(\zeta^{\frac12}-\zeta^{-\frac12}\right)^2\exp\!\left(\frac{4\pi^2 w^2}{z}\right)\eta\! \left(\frac{2\pi i}{z}\right)^{6}}{\left(\frac{z}{2\pi}\right)^{2}\vartheta\!\left(\frac{2\pi w}{iz}; \frac{2\pi i}{z}\right)^2}.
\]
Then expanding the definitions one obtains
\begin{eqnarray*}
g\!\left(w; \frac{iz}{2\pi}\right)^2&=-\frac{\left(\zeta^{\frac12}-\zeta^{-\frac12}\right)^2\exp\left(\frac{4\pi^2 w^2}{z}\right)}{\left(\frac{z}{2\pi}\right)^{2}\left(e^{\frac{2\pi^2 w}{z}}-e^{-\frac{2\pi^2 w}{z}}\right)^2 } \prod_{n\geq 1}\frac{\left(1-e^{-\frac{4\pi^2 n}{z}}\right)^{4}}{\left(1-e^{\frac{4\pi^2 w}{z}-\frac{4\pi^2 n}{z}}\right)^2\left(1-e^{-\frac{4\pi^2 w}{z}-\frac{4\pi^2 n}{z}}\right)^2}\\ 
&=\frac{\sin(\pi w)^2\exp\left(\frac{4\pi^2 w^2}{z}\right)}{\left(\frac{z}{2\pi}\right)^2\sinh\!\left(\frac{2\pi^2 w}{z}\right)^2}
\left(1+O\!\left(e^{-4\pi^2\text{Re}\left(\frac1z\right)(1-w)}\right)\right),
\end{eqnarray*}
which completes the proof.
\end{proof}

We continue by using Cauchy's theorem to  express the coefficients of $f_k(w;\tau)$ as contour integrals. We define $f_{m,k}(\tau)$ as the coefficient of $\zeta^m$, 
\be \non
f_{m,k}(\tau):=\frac{2\,q^{\frac{k}{24}}}{\eta(\tau)^k}\int_0^{\frac{1}{2}} g(w;\tau)^2\,\cos(2\pi m w)dw,
\ee 
where we used that $f_k(-w;\tau)=f_k(w;\tau)$. From $f_{m,k}(\tau)$, the $a_{m,k}(n)$ are consequently obtained as
\be
\label{eq:amkn}
a_{m,k}(n):=\frac{1}{2\pi i}\int_{\mathcal{C}} \frac{f_{m,k}(\tau)}{q^{n+1}}dq,
\ee  
where $\mathcal{C}$ is a circle surrounding 0 clockwise. We choose $e^{-\beta_k}$ for the radius, with $\beta_k:= \pi\sqrt{\frac{k}{6n}}$. Lemmas \ref{modular} and \ref{approx}  show that in order to obtain the asymptotic main term of $a_{m,k}(n)$, the following split is natural \cite{Wright:1971, Bringmann:2013yta}
\[
a_{m,k}(n)=M + E,
\] 
with
\begin{align*}
M&:= \frac1{2\pi i}\int_{\mathcal{C}_1}\frac{f_{m,k}(\tau)}{q^{n+1}}dq,\\
E&:= \frac1{2\pi i}\int_{\mathcal{C}_2}\frac{f_{m,k}(\tau)}{q^{n+1}}dq,
\end{align*}
where $\CC_1$ is the arc going counterclockwise from phase $-\beta_k$ to $\beta_k$, and $\CC_2$ is the complement of $\CC_1$ in $\CC$. 

The leading term will follow from $M$, whereas $E$ will contribute to the error.
We first consider $M$ and split this further into
$$
M=M_1+E_1,
$$
with 
\begin{eqnarray}
\label{eq:M2}
&&M_1 := \frac{1}{2\pi i} \int_{\mathcal{C}_1} \frac{g_{m,1}(z)}{\left(e^{-z};e^{-z}\right)^k_\infty}\, q^{-(n+1)}\, dq, \\
\label{eq:E1}
&&E_1 := \frac{1}{2\pi i} \int_{\mathcal{C}_1} \frac{g_{m,2}(z)}{\left(e^{-z};e^{-z}\right)^k_\infty}\, q^{-(n+1)} dq,
\end{eqnarray}
and $g_{m,1}(z)$ and $g_{m,2}(z)$ are defined by
\begin{eqnarray}\label{defineGm}
g_{m,1}(z) &:=& \frac{8\pi^2}{z^2}\int_0^{\frac12}\frac{\sin(\pi
  w)^2}{\sinh\!\left(\frac{2\pi^2 w}{z}\right)^2} e^{\frac{4\pi^2
    w^2}{z}} \cos(2\pi m w) dw,\\
g_{m,2}(z) &:=&  2\int_0^{\frac12} \left( g\! \left( w; \frac{iz}{2\pi}\right)^2 - \frac{ \sin(\pi w)^2}{\left(\frac{z}{2\pi}\right)^2 \sinh\left(\frac{2 \pi^2 w}{z}\right)^2} e^{\frac{4\pi^2 w^2}{z}} \right) \cos(2\pi  mw) dw. \non
\end{eqnarray}
In view of Lemma \ref{approx} this is the natural splitting.

Continuing in analogy with \cite{Bringmann:2013}, we insert the Taylor expansions of $\sin^2(\pi w)$, $\cos(2\pi m w)$, and $\exp\!\left(\frac{4\pi^2w^2}{z}\right)$, 
\begin{eqnarray}
&&\sin(\pi w)^2=-\frac14\left(e^{2\pi iw}+e^{-2\pi iw}-2\right)=-\frac12 \sum_{\ell\geq 1}(-1)^\ell\frac{(2\pi w)^{2\ell}}{(2\ell)!}, \non \\
&&\cos(2\pi m w)=\sum_{\ell \geq 0} (-1)^\ell \frac{(2\pi m w)^{2\ell}}{(2\ell)!}, \non \\
&&\exp\!\left(\frac{4\pi^2w^2}{z}\right)=\sum_{j\geq 0}\frac{\left(\frac{4\pi^2w^2}{z}\right)^j}{j!},\non
\end{eqnarray}
into equation \eqref{defineGm}. Since the Taylor series converge absolutely, we can exchange the sums and the integral, such that we find for $g_{m,1}(z)$
\begin{equation*}
g_{m,1}(z)= \sum_{\ell_1\geq 1,\ell_2\geq 0,\atop  j\geq 0}(-1)^{\ell_1+\ell_2+1} \frac{(2\pi)^{2(\ell_1+\ell_2+j+1)}}{(2\ell_1)!\,(2\ell_2)!\,j!} \frac{m^{2\ell_2}}{z^{j+2}} \int_0^{\frac12}\frac{w^{2(\ell_1+\ell_2+j)}}{\sinh\!\left(\frac{2\pi^2w}{z}\right)^2}dw.
\end{equation*} 

We are thus left to evaluate integrals of the shape $(j\in\N^*)$
\begin{equation}\label{coshint}
\int_0^{\frac12}\frac{w^{2j}}{\sinh\!\left(\frac{2\pi^2w}{z}\right)^2}dw.
\end{equation}
This integral is convergent since the integrand behaves as $w^{2j-2}$ as function of $w$. We next extend the integration domain $[0,\frac{1}{2}]$ to $[0,\infty]$. Using the incomplete Gamma function $\Gamma\left(j;x\right):= \int_{x}^{\infty} e^{-t}t^{j-1}dt$, the error may be bounded by
\be
\label{eq:errorint}
\ll \int_{\frac12}^\infty w^{2j} e^{-4\pi^2 w\text{Re}\left(\frac1z\right)}dw \ll \left(\textrm{Re}\!\left(\frac{1}{z}\right)\right)^{-2j-1}\Gamma\left(2j+1; 2\pi^2\, \textrm{Re}\!\left(\frac{1}{z}\right)\right)\ll e^{-2\pi^2\, \textrm{Re}\left(\frac{1}{z}\right)},
\ee
where we throughout use that $g(x)\ll f(x)$ means $g(x)=O(f(x))$, and the well known fact that 
$$
\Gamma\left(j;x\right) \sim x^{j-1}\, e^{-x},
$$
as $x \rightarrow \infty$. In the new integral we make the change of variables $\frac{2\pi w}{z}=u$. The path then is given by $\text{Arg}(u)=\text{Arg}(z)$.
Using the Residue Theorem we can shift the path of integration down to $\R$ giving
\[
\left(\frac{z}{2\pi}\right)^{2j+1}\int_0^\infty \frac{u^{2j}}{\sinh(\pi u)^2}du.
\]
Now define
\[
\mathcal{B}_j:=\int_0^\infty \frac{u^{2j}}{\sinh(\pi u)^2}du.
\]
We will need the following evaluation
\begin{lemma}\label{Inter}
We have
\[
\mathcal{B}_j= \frac{(-1)^{j+1} B_{2j}}{\pi},
\]
where $B_j$ denotes the $j$-th Bernoulli number.
\end{lemma}

\begin{proof}
The proof is similar in spirit to \cite[Lemma 5.2]{Bringmann:2012}. We first extend the integral to
$\mathbb{R}$. Since the poles all lie at $i\mathbb{Z}/\{0\}$, we can 
shift the integral away from the real axis. One obtains
$$
\mathcal{B}_j=\frac{1}{2} \int_{\mathbb{R}+\frac{i}{2}} 
\frac{u^{2j}}{\sinh(\pi u)^2}du.
$$
Define the function $g(u,T):=\frac{e^{2\pi i Tu}}{\sinh(\pi
  u)^2}$. Its residue is given by
$$
2\pi i\,
\mathrm{Res}_{u=i}(\,g(u,T)\,)=\left(\int_{\mathbb{R}+\frac{i}{2}}-\int_{\mathbb{R}+\frac{3i}{2}}\right)g(u,T)du=-4Te^{-2\pi
T}.
$$
Moreover, we have that 
$\int_{\mathbb{R}+\frac{3i}{2}}g(u,T)du=\int_{\mathbb{R}+\frac{i}{2}}g(u+i,T)du=e^{-2\pi
T} \int_{\mathbb{R}+\frac{i}{2}}g(u,T)du$, which gives us the
integral
$$
\int_{\mathbb{R}+\frac{i}{2}}g(u,T)du=\frac{4T}{1-e^{2\pi
    T}}.
$$
The generating function of the Bernoulli numbers $B_m$
\begin{equation}\label{Bernoulli}
\frac{x}{e^x-1}=\sum_{m=0}^\infty B_m \frac{x^m}{m!},
\end{equation}
and the expansion of the numerator of $\frac{e^{2\pi i Tu}}{\sinh(\pi
  u)^2}$ in the integral gives the desired result.
\end{proof}

Now combining Lemma \ref{Inter} with the error (\ref{eq:errorint}), we have
\begin{eqnarray}
\label{eq:Taylorgz}
g_{m, 1}(z)&=&2\sum_{\ell_1\geq 1,\ell_2\geq 0, \atop j\geq 0} (-1)^{j} \frac{m^{2\ell_2}\, }{(2\ell_1)!\,(2\ell_2)!\,j!}\,z^{2(\ell_1+\ell_2)+j-1}\\
&&\times\left(B_{2(\ell_1+\ell_2+j)}+O\!\left(|z|^{-2(\ell_1+\ell_2+j)-1}e^{-2\pi^2\mathrm{Re}\left(\frac{1}{z}\right)}\right)\right).\non
\end{eqnarray}

To evaluate the integral $M_1$ defined in (\ref{eq:M2}), we can proceed as in \cite{Bringmann:2013yta, Wright:1971}. First we define the coefficients $d_{m,k}(\ell)$ as the Taylor coefficients of $g_{m,1}(z)$, where we let $N \in \N$,
\begin{eqnarray}
\label{eq:dmkz}
e^{-\frac{kz}{24}} g_{m,1}(z)=:\sum_{\ell=1}^N d_{m,k}(\ell)\,z^\ell +O\!\left(z^{N+1}\right).
\end{eqnarray}
The first few coefficients are given by
\be
d_{m,k}(1)=\frac{1}{6}, \qquad d_{m,k}(2)=\frac{1}{30}-\frac{k}{144},\qquad d_{m,k}(3)=\frac{23}{2520}-\frac{m^2}{60}-\frac{k}{720}+\frac{k^2}{6912}. \non
\ee
Having obtained equation (\ref{eq:dmkz}), we make two further splits, where the first one is natural in view of Lemma \ref{modular}
$$
M_1 = M_2 + E_2,
$$
with
\begin{align*}
 M_2 & := \frac{1}{2\pi i}\int_{\mathcal{C}_1}\frac{g_{m,1}(z)}{q^{n+1}}\left(\frac{z}{2\pi}\right)^{\frac{k}{2}}e^{-\frac{kz}{24}+\frac{k\pi^2}{6z}}dq, \\
 E_2 & := \frac{1}{2\pi i}\int_{\mathcal{C}_1}\frac{g_{m,1}(z)}{q^{n+1}}\left(\frac{1}{(e^{-z};e^{-z})_{\infty}^k}-\left(\frac{z}{2\pi}\right)^{\frac{k}{2}}e^{-\frac{kz}{24}+\frac{k\pi^2}{6z}}\right)dq,
\end{align*}
and
$$
M_2 = M_3 + E_3,
$$
where
\begin{align*}
 M_3 &:= \frac{1}{2\pi i}\sum_{\ell=1}^{N} d_{m,k}(\ell)\int_{\mathcal{C}_1}\frac{1}{q^{n+1}}\left(\frac{z}{2\pi}\right)^{\frac{k}{2}} e^{\frac{k\pi^2}{6z}}z^{\ell}dq, \\
 E_3 &:=  \frac{1}{2\pi i} \int_{\mathcal{C}_1}\frac{1}{q^{n+1}}\left(\frac{z}{2\pi}\right)^{\frac{k}{2}} e^{\frac{k\pi^2}{6z}}\left(e^{-\frac{kz}{24}}g_{m,1}(z) - \sum_{\ell =1}^{N}d_{m,k}(\ell)z^{\ell}\right)dq.
\end{align*}

After a change of the integration variable $v=z/\beta_k$, the main term $M_3$ consists of contour integrals of the form
\be \non
\mathcal{I}_s(\alpha):=\frac{1}{2\pi i} \int_{1-i}^{1+i} v^se^{\alpha\left(v+\frac{1}{v} \right)}dv, \qquad \alpha>0.
\ee
Lemma 3.1 in \cite{Bringmann:2013yta} estimates this integral in terms of the $I$-Bessel function
\be
\label{eq:IksBessel} \non
\mathcal{I}_s(\alpha)= I_{-s-1}\!\left(2\alpha\right) + O\!\left( e^{\frac{3}{2}\alpha} \right),
\ee
with
\be
\label{eq:IBessel} 
I_\ell(2\sqrt{z})=\frac{z^{\frac{\ell}{2}}}{2\pi i}\int_{-\infty}^{(0+)} t^{-\ell-1} \exp\!\left(t+\frac{z}{t}\right)dt,
\ee
where the integral is along any counterclockwise contour looping from $-\infty$ around $0$ back to $-\infty$. 
Theorem 1.1 follows now from substitution of these expressions and $\alpha=\beta_k$, except for the determination of the error terms, which we bound in the following subsection. Corollary \ref{cor1} follows from Theorem 1.1 and using the asymptotic behavior of the Bessel function \cite{Andrews99}
\be
\label{eq:approxBessel}
I_{s}(x) = \frac{e^x}{\sqrt{2\pi x}} + O\!\left(\frac{e^x}{x^{\frac{3}{2}}}\right). 
\ee

\subsection{The error term}
We determine in this subsection the magnitude of the error terms $E,E_1,E_2,$ and $E_3$, and that they are ignorable compared to the main term. We show that the main error is due to $E_3$. We start by computing bounds for the error terms coming from the different approximations near the dominant pole starting with $E_1$. To this end we first bound $g_{m,2}(z)$, which is given by \eqref{defineGm}, near the dominant pole

\begin{lemma}\label{errorterm}
We have for $z\in \CC_1$ and $\beta_k\to 0$
$$
g_{m,2}(z) \ll \frac{e^{-\frac{3\pi^2}{2 \beta_k}}}{\beta_k^2}.
$$
\end{lemma}
\begin{proof}
Recall that $\CC_1$ is the arc for $q=e^{-z}$ with phase going from $-\beta_k$ to $\beta_k$ and radius $e^{-\beta_k}$. One straightforwardly establishes the following bounds on this arc
\begin{equation} \non
 \qquad |z| \gg \beta_k, \qquad 
 \textrm{Re}\!\left(\frac{1}{z}\right)  \geq \frac{1}{2 \beta_k}.
\end{equation}
Furthermore, the quotient
$$
\left|\frac{\sin(\pi w)}{1-e^{-\frac{4\pi^2w}{z}}}\right| \ll 1,
$$
is bounded for $w\in \left[0,\frac{1}{2}\right]$, since the numerator is obviously bounded and for $\left| \frac{w}{z}\right|\ll 1$, the quotient is $\ll |z| $, and for larger $w$ the denominator is $\ll 1$ for $\beta_k\to 0$. Using Lemma \ref{approx} and the bounds above, we bound $g_{m,2}(z)$ by
\begin{align*}
g_{m,2}(z) & \ll \frac{1}{|z|^2} \int_{0}^{\frac{1}{2}} \left| \frac{ \sin(\pi w)^2}{\left(1 - e^{-\frac{4\pi^2 w}{z}}\right)^2}\right| e^{4 \pi^2\textrm{Re}\left(\frac{1}{z}\right)(w^2 -1 )} dw 
         \ll \frac{e^{-\frac{3\pi^2}{2 \beta_k}}}{\beta_k^2 },
\end{align*}
where we used that $w^2 -1$ has its maximum on $\left[0,\frac{1}{2}\right]$ at $\frac{1}{2}$. 
\end{proof}

With this result we find for $E_1$, defined in equation (\ref{eq:E1}),
\begin{lemma}
\label{E1}
We have for $n \rightarrow \infty$
$$
E_1 \ll n^{-\frac{k}{4} + \frac{1}{2}}e^{\pi\sqrt{\frac{2kn}{3}} - \frac{3}{2}\pi\sqrt{\frac{6n}{k}}}.
$$
\end{lemma}
\begin{proof}
Using the bound of $g_{m,2}(z)$ (see Lemma \ref{errorterm}) we obtain directly
$$
E_1 \ll \frac{e^{-\frac{3\pi^2}{2\beta_k}}}{\beta_k^2}\int_{\mathcal{C}_1}\frac{q^{-(n+1)}}{\left(e^{-z};e^{-z}\right)_{\infty}^k}dq.
$$
From Lemma \ref{modular} it is easy to see that
%
%
%
\be
\label{eq:qpochk} 
\frac{1}{(e^{-z};e^{-z})^k_\infty}=\left(\frac{z}{2\pi}\right)^{\frac{k}{2}}\, e^{-\frac{kz}{24}+\frac{k\pi^2}{6z}} \left(1+O\!\left(e^{-\frac{4\pi^2}{z}}\right)\right).
\ee
As a result we find for $n\to \infty$
\begin{equation*}
E_1 \ll\beta_k^{\frac{k}{2}-2}e^{-\frac{3\pi^2}{2\beta_k}}\int_{-\beta_k}^{\beta_k} e^{(n-\frac{k}{24})\beta_k + \frac{k \pi^2}{6}\text{Re}\left(\frac{1}{z}\right)}dz.\\
\end{equation*}
Now we investigate the exponent, which can be rewritten and bounded by
$$
\pi \sqrt{\frac{2kn}{3}} - \frac{3}{2}\pi\sqrt{\frac{6n}{k}}. 
$$ 
This follows from the following upperbound for $\re\!\left( \frac{1}{z}\right)$ on $\mathcal{C}_1$
$$
 \frac{\textrm{Re}(z)}{|z|^2} \leq \frac{1}{\textrm{Re}(z)} = \frac{1}{\beta_k},
$$
and so
\begin{equation}\label{eq:useful} 
\left(n-\frac{k}{24}\right)\beta_k + \frac{\pi^2 k}{6}\textrm{Re}\!\left(\frac{1}{z}\right) <\pi \sqrt{\frac{2kn}{3}}
\end{equation}
by substituting $\beta_k=\pi\sqrt{\frac{k}{6n}}$. Since the length of the integration path is of order $O(\beta_k)$, we arrive at the desired result.
\end{proof}
The next step is to evaluate the error $E_2$ coming from approximation of the $q$-Pochhammer symbol by its functional equation.
\begin{lemma} 
\label{E2}
We have for $n \rightarrow \infty$
$$
E_2 \ll n^{-\frac{k}{4}-1}e^{\pi\sqrt{\frac{2kn}{3}}-4\pi\sqrt{\frac{6n}{k}}}.
$$
\end{lemma}
\begin{proof}
On $\mathcal{C}_1$ the following approximation is valid 
$$
 |z|^2 = \beta_k^2 +  \textrm{Im}(z)^2 \leq 2\beta_k^2.
$$
Since the leading term in the Taylor series of $g_{m,1}(z)$, given in equation \eqref{eq:Taylorgz}, is $\frac{z}{6}$, we can bound $g_{m,1}(z)$ as
$$
g_{m,1}(z) \ll |z| \ll \beta_k.
$$
From (\ref{eq:qpochk}) we know that 
$$
\frac{1}{(e^{-z};e^{-z})_{\infty}^k} - \left(\frac{z}{2\pi}\right)^{\frac{k}{2}}e^{-\frac{kz}{24}+\frac{k\pi^2}{6z}} = O\!\left(z^\frac{k}{2}\,e^{-\frac{kz}{24}+\frac{4\pi^2}{z}(\frac{k}{24}-1)}\right).
$$
Now we have
$$
E_2 \ll \int_{\mathcal{C}_1} dz\, e^{n\beta_k+\frac{\pi^2k}{6\beta_k}-\frac{4\pi^2}{\beta_k}} \beta_k^{\frac{k}{2}+1}.
$$
By noting that the integration path is of order $O(\beta_k)$ and plugging in $\beta_k$ we obtain the lemma.
\end{proof}
To finish the analysis of the error terms on the major arc we calculate $E_3$ coming from the replacement of our main term by a Taylor series. 
\begin{lemma}
We have, as $n \rightarrow \infty$
$$
E_3 \ll n^{-1-\frac{N}{2}-\frac{k}{4}}e^{\pi\sqrt{\frac{2kn}{3}}}.
$$
\end{lemma}
\begin{proof}
Using
$$
e^{-\frac{kz}{24}}g_{m,1}(z) - \sum_{\ell =1}^{N}d_{m,k}(\ell)z^{\ell} = O(z^{N+1}),
$$
and changing variables we have
$$
E_3 \ll \int_{\mathcal{C}_1} |z|^{\frac{k}{2} +N +1}e^{n\beta_k+\frac{\pi^2 k}{6}\textrm{Re}\left(\frac{1}{z}\right)}dz.
$$
Using \eqref{eq:useful}, $|z| \ll \beta_k$ and that the path is of order $O(\beta_k)$ gives the desired result.
\end{proof}
To obtain an error term away from the dominant pole (also known as the minor arc) we use the following Lemma proved in \cite{Bringmann:2013}.
\begin{lemma}\label{LemmaBrDo}
Assume that $\tau = u + iv \in \H$ with $Mv \leq |u| \leq \frac{1}{2}$ for $u>0$ and $v \rightarrow 0$, we have that
$$
\vert (q;q)_\infty^{-1}\vert \ll \sqrt{v}\exp\left[\frac{1}{v}\left(\frac{\pi}{12}- \frac{1}{2\pi}\left(1 - \frac{1}{\sqrt{1 + M^2}}\right)\right)\right].
$$
\end{lemma}
This means that the contribution of the other roots of unity will be suppressed as we see by bounding the error term $E$.
\begin{lemma} 
 We have, for every $0<\varepsilon \leq 1$, as $ n \rightarrow \infty$ 
$$
E \ll n^{-\frac{k-6}{4}} e^{\pi\sqrt{\frac{2kn}{3}}\left(1-\frac{3\varepsilon}{4\pi^2}\right)}.
$$
\end{lemma}
\begin{proof}
We first bound $g(w;\tau)$. To this end, we write $g(w;\tau)$ as a sum over its poles \cite{AG}
$$
g(w;\tau) = 1 + \left(1-\zeta\right)\sum_{m \geq 1} \frac{(-1)^m q^{\frac{m^2 + m}{2}}}{1-\zeta q^m} + \left(1-\zeta^{-1}\right)\sum_{m \geq 1} \frac{(-1)^m q^{\frac{m^2 + m}{2}}}{1-\zeta^{-1}q^m}.
$$
Therefore $g(w;\tau)$ can be bounded for $\im(\tau)=\frac{\beta_k}{2\pi}$, $\im(w)=0$ and $n\to \infty$ as follows
$$
g(w;\tau) \ll \sum_{m\geq1}\frac{\vert q\vert^{\frac{m^2+m}{2}}}{1-\vert q \vert^m} \ll \frac{1}{1-\vert q \vert} \sum_{m\geq1}e^{-\frac{\beta_k m^2}{2}} \ll \beta_k^{-\frac{3}{2}} \ll n^{\frac{3}{4}}, 
$$
where the second last bound comes from comparison with a gaussian integral. 
Thus
\begin{equation}\label{eq:gsquare} \non
g(w;\tau)^2 \ll n^{\frac{3}{2}}.
\end{equation}
We use Lemma \ref{LemmaBrDo} for the arc $\CC_2$. Recall that $\tau = \frac{iz}{2\pi}$ and $\textrm{Re}(z)= \beta_k = \pi\sqrt{\frac{k}{6n}}$. Consequently, $v=\frac{\beta_k}{2\pi}$ and $M$ in Lemma \ref{LemmaBrDo} equals 1. Using this and the bound for $g(w;\tau)$, we directly obtain 
$$
E \ll n^{\frac{3}{2}}\int_{\CC_2} \beta_k^\frac{k}{2} e^{n\beta_k+n \beta_k\left(1 -\frac{6}{\pi^2}\left(1-\frac{1}{\sqrt{2}} \right)\right)} \, dz.
$$
Using that for $n\to \infty$, $ -\frac{6}{\pi^2}n\beta_k\,(1-\frac{1}{\sqrt{2}})<-\frac{6}{4\pi^2}n\beta_k$ and that the integration path is $O(1)$ finishes the proof.
\end{proof}
\vspace{.3cm}
This finishes the proof of Theorem \ref{th1}, since we have computed the main term $M_3$ and determined that the leading error among the error terms $E$, $E_i$, $i=1,2,3$ is given by $E_3$.
\section{Proof of Theorem \ref{th2}}\label{sec:th2}
In this section we calculate the main term that contributes to the profile coming from approaching the main singularity.  To deduce Theorem \ref{th2} we then use Wright's circle method. Moreover we detect the error coming from terms near the dominant pole and away from the dominant pole, giving the range where the asymptotic expansion is valid.

\subsection{The main term}

To determine the profile of $a_{m,k}(n)$ as a function of $m$ for large $n$, we start by determining an expansion of $g_{m,1}(z)$ which is valid for a wide range of $m$. The range of $m$ is $1\leq |m|\leq\frac{1}{6\beta_k}\log n$. One verifies that with this range the error in Theorem \ref{th2} goes to zero for large $n$. Furthermore, we set $z=\beta_k\left(1 + i u m^{-\frac{1}{3}}\right)$ for $|m|\geq 1$ and as before $\beta_k:= \pi \sqrt{\frac{k}{6 n}}$. 

To prove Theorem \ref{th2} we continue in much the same way as in Section \ref{th1} using the approach of \cite{Bringmann:2013, Bringmann:2013yta} to perform Wright's variant \cite{Wright:1971} of the circle method. We recall the definition of $a_{m,k}(n)$ (\ref{eq:amkn})
$$
a_{m,k}(n):= \frac{1}{2\pi i }\int_{\mathcal{C}}\frac{f_{m,k}(q)}{q^{n+1}}dq,
$$
where the contour is as in Section \ref{sec:th1} the counterclockwise transversal of the circle $C:= \lbrace q \in \C; |q|=e^{-\beta_k} \rbrace$. We change variables to $z=\beta_k(1+ium^{-\frac{1}{3}})$ and obtain
$$
a_{m,k}(n) = \frac{\beta_k}{2 \pi m^{\frac{1}{3}}}\int_{\mathcal{D}} f_{m,k}(e^{-z})\,e^{nz}\,du,
$$
where $\mathcal{D}$ is the interval $u\in \left[-\frac{\pi m^{\frac{1}{3}}}{\beta_k}, \frac{\pi m^{\frac{1}{3}}}{\beta_k}\right]$. We split as before
$$
a_{m,k}(n)= M + E,
$$
with
\begin{align*}
M &:= \frac{\beta_k}{2 \pi m^{\frac{1}{3}}}\int_{\mathcal{D}_1} f_{m,k}(e^{-z})e^{nz}du, \\
E &:= \frac{\beta_k}{2 \pi m^{\frac{1}{3}}}\int_{\mathcal{D}_2} f_{m,k}(e^{-z})\,e^{nz}\,du,
\end{align*}
where $\mathcal{D}_1$ is the interval $u\in \left[ -1,1\right]$ and $\mathcal{D}_2$ is the complement of $\mathcal{D}_1$ in $\mathcal{D}$. Completely analogously to Section \ref{sec:th1} we split $M=M_1+E_1$ and $M_1=M_2+E_2$, where $M_2$, $E_1$ and $E_2$ are now defined as
\begin{eqnarray}
&&M_2:=\frac{\beta_k}{2\pi m^\frac{1}{3}} \int_{\cD_1} g_{m,1}(z) \,\left( \frac{z}{2\pi}\right)^\frac{k}{2}\, e^{-\frac{kz}{24}+\frac{k\pi^2}{6z}+nz}\, du, \non \\
&&E_1:=\frac{\beta_k}{2\pi m^\frac{1}{3}} \int_{\cD_1} \frac{g_{m,2}(z)}{\left(e^{-z};e^{-z} \right)^k_\infty}\, e^{nz}\, du, \non \\
&&E_2:=\frac{\beta_k}{2\pi m^\frac{1}{3}} \int_{\cD_1} g_{m,1}(z) \,\left( \frac{1}{\left(e^{-z};e^{-z} \right)^k_\infty}-\left( \frac{z}{2\pi}\right)^\frac{k}{2}\, e^{-\frac{kz}{24}+\frac{k\pi^2}{6z}}\right)\,e^{nz}\, du. \non
\end{eqnarray} 

In the following Lemma we give an approximation for $g_{m,1}(z)$ as $z\to 0$, which is valid for the wide range of $m$ mentioned above. We resum the sum over $\ell_2$ (the exponents of $m$) in the Taylor series for $g_{m,1}(z)$ (\ref{eq:Taylorgz}). This gives an expression in terms of hyperbolic trigonometric functions
\begin{lemma}\label{mainterm}
Recall that $P(m,\beta)$ is defined in equation \eqref{eqn:profile}. Assume that $|u|\leq 1$ and $ m \leq \frac{1}{6 \beta_k}\log n$. Then we have as $ n \rightarrow \infty$
$$
g_{m, 1}(z) = \left(1 + i u m^{-\frac{1}{3}}\right) P(m,\beta_k) +  O\!\left(\beta_km^{\frac{2}{3}}P\!\left(m,\beta_k\right)\right).
$$
\end{lemma}

\begin{proof}
Recall that we determined in the previous section the Taylor series for $g_{m, 1}(z)$ defined in equation (\ref{eq:Taylorgz}). Using the generating function of the Bernoulli numbers, given by equation (\ref{Bernoulli}), we can approximate this as
\begin{align*}
 g_{m, 1}(z)= & \sum_{\ell=0}^{\infty} \frac{(mz)^{2\ell}}{(2\ell)!}\left(z B_{2\ell +2} + O(|z|^2) \right) \\ 
                  = & z \frac{d^2}{d(mz)^2} \left(\frac{mz}{2}\coth\left(\frac{mz}{2}\right) \right) + O\!\left(|z|^2 \cosh(mz)\right) \\
= & P(m,z)+ O\!\left(|z|^2 \cosh(mz)\right).
\end{align*}

We note that $z^{-1}\,P(m,z)=:f(mz)$ is only a function of $mz$. This function $f(x)$ is clearly smooth for $x\neq 0$, and one easily verifies that $f(x)$ is also analytic for $x=0$
$$
f(x)=\frac{1}{6}+ O\!\left( x^2 \right).
$$
We can thus make a Taylor expansion of $f(x)$ around any $x\in \mathbb{R}$. Since $|f'(x)|\leq |f(x)|$, we have
\be
f(x+\varepsilon)=f(x)+f'(x)\varepsilon + O(\varepsilon^2)=f(x)+O(f(x)\varepsilon). \non
\ee
Now we apply this to $g_{m,1}(z)$ with $x=\beta_k m$ and $\varepsilon=i \beta_k m^{\frac{2}{3}}$
$$
g_{m,1}(z)=z\, f(m\beta_k)+O\!\left(\beta_k^2\, m^{\frac{2}{3}}\, (1+m^{-\frac{2}{3}})^\frac{1}{2} f(m\beta_k) \right)+O\!\left(\beta_k^2 (1+m^{-\frac{2}{3}})\cosh(m\beta_k) \right).
$$

Now we show that the first error term is larger than the second term for the full range of $m$. For that we distinguish between the cases where $\beta_km$ is bounded and where $\beta_k m$ grows as $\frac{1}{6}\log(n)$. If $\beta_k m$ is bounded, both $f(\beta_k m)$ and $\cosh(\beta_k m)$ are $O(1)$. Since in this case $m=O\!\left(n^\frac{1}{2}\right)$ for $n\to \infty$, we find thus that the first error is $O\!\left(n^{-\frac{2}{3}}\right)$ and the second error term is $O\!\left(n^{-1}\right)$. For $\beta_k m \to \frac{1}{6}\log(n)$ as $n\to \infty$, we can bound the $\cosh(\beta_k m)$ by
$$
\cosh(\beta_k m) \ll e^{\beta_k m} \ll n^{\frac{1}{6}}.
$$
Similarly, one finds for $f(\beta_km)$
$$
f(\beta_km)\ll n^{-\frac{1}{6}}\log(n).
$$
As a result, the first error becomes $O\!\left(n^{-\frac{5}{6}}\log(n)^\frac{5}{3}\right)$ and the second error $O\!\left(n^{-\frac{5}{6}}\right)$. This concludes the proof of Lemma \ref{mainterm}.  \end{proof}

This Lemma leads us to the last split $M_2=M_3+E_3$ with
\begin{eqnarray}
&&M_3:=\frac{1}{2\pi m^\frac{1}{3}} \int_{\cD_1} z\, P(m,\beta_k) \,\left( \frac{z}{2\pi}\right)^\frac{k}{2}\, e^{-\frac{kz}{24}+\frac{k\pi^2}{6z}+nz}\, du,  \non \\
&&E_3:= \frac{\beta_k}{2\pi m^\frac{1}{3}} \int_{\cD_1} \left(g_{m,1}(z) - \frac{z}{\beta_k} P(m,\beta_k) \right) \,\left( \frac{z}{2\pi}\right)^\frac{k}{2}\, e^{-\frac{kz}{24}+\frac{k\pi^2}{6z}+nz}\, du. \non 
\end{eqnarray}
In order to determine the main term, we first define the following function
$$
\mathcal{J}_{s}(\alpha):=\frac{1}{2\pi i} \int_{1-i m^{-\frac{1}{3}}}^{1+i m^{-\frac{1}{3}}} v^s e^{\alpha\left(v + \frac{1}{v}\right)}dv,\qquad \alpha>0,
$$ 
and recall that Ref. \cite[Lemma 4.2]{Bringmann:2013} shows that these integrals may be related to $I$-Bessel functions (analogously to $\mathcal{I}_s(\alpha)$ in Section \ref{sec:th1})
\begin{lemma}\label{pasy} 
As $ n \rightarrow \infty$
$$
\mathcal{J}_s(\alpha)= I_{-s-1}\!\left( 2\alpha\right) + O\!\left( \exp\!\left(\alpha \left(1+\frac{1}{1+ m^{-\frac{2}{3}}}\right)\right)\right).
$$
\end{lemma}
With this lemma we prove the the following proposition for $M_3$
\begin{proposition} 
 We have 
$$
M_3 = P(m,\beta_k)\,p_k(n)\left(1 + O\!\left(n^{-\frac{1}{2}}\right)\right).
$$
\end{proposition}
\begin{proof}
The change of variables $v=1+ium^{-\frac{1}{3}}$ gives
$$
 M_3= \frac{\beta_k^{\frac{k}{2}+1}}{(2\pi)^{\frac{k}{2}}} P(m,\beta_k) \, \frac{1}{2\pi i} \int_{1-im^{-\frac{1}{3}}}^{1+im^{-\frac{1}{3}}} v^{\frac{k}{2}+1}e^{\pi v(n-\frac{k}{24})\sqrt{\frac{k}{6n}}+\frac{\pi}{v}\sqrt{\frac{nk}{6}}}dv.
$$
We approximate the integral over $v$ for $n\to \infty$ by
$$
\frac{1}{2\pi i} \int_{1-im^{-\frac{1}{3}}}^{1+im^{-\frac{1}{3}}} v^{\frac{k}{2}+1}e^{\pi v\sqrt{\frac{kn}{6}}+\frac{\pi}{v}\sqrt{\frac{kn}{6}}}\left( 1+ O\!\left(n^{-\frac{1}{2}}v \right)\right)dv.
$$
Now using the definition of $\mathcal{J}_s(\alpha)$ this equals
\be
\label{eq:Js}
\mathcal{J}_{\frac{k}{2}+1}\!\left(\pi\sqrt{\frac{kn}{6}} \right)+O\!\left(n^{-\frac{1}{2}}\mathcal{J}_{\frac{k}{2}+2}\!\left(\pi\sqrt{\frac{kn}{6}} \right) \right).
\ee
Using Lemma \ref{pasy} and the asymptotic expansion of the Bessel function given in equation (\ref{eq:approxBessel}),
equation (\ref{eq:Js}) is further approximated by
$$ 
\frac{e^{\pi\sqrt{\frac{2k n}{3}}}}{\pi\sqrt{2}\left(\frac{2}{3}k n\right)^{\frac{1}{4}}} + O\!\left(\frac{e^{\pi\sqrt{\frac{2k n}{3}}}}{n^{\frac{3}{4}}}\right) + O\!\left( \exp\left(\pi\sqrt{\frac{k n}{6}}\left(1+\frac{1}{1+ m^{-\frac{2}{3}}}\right)\right)\right).
$$
One easily sees that the first error term is the largest one and so we have for the leading term $M_3$
$$
M_3 = \frac{\beta_k^{\frac{k}{2}+1}}{(2\pi)^{\frac{k}{2}}}\,P(m,\beta_k )\frac{e^{\pi\sqrt{\frac{2k n}{3}}}}{\pi\sqrt{2}(\frac{2}{3}kn)^{\frac{1}{4}}}\left( 1 + O\!\left( n^{-\frac{1}{2}}\right)\right).
$$
Now using the following well-known formula \cite{HR2, RademacherZ:1938} for $p_k(n)$, for $n \rightarrow \infty$,
$$
p_k(n) = 2 \left(\frac{k}{3}\right)^{\frac{k+1}{4}}(8n)^{-\frac{k+3}{4}}e^{\pi\sqrt{\frac{2kn}{3}}}\left(1 + O\!\left(  n^{-\frac{1}{2}} \right)\right),
$$
we finish the proof.

\end{proof}

\subsection{The error term}
In this subsection, we discuss the error terms $E_1$, $E_2$, $E_3$ near the dominant pole. We also determine the error $E$ to Theorem \ref{th2} away from the dominant pole, which is due to the minor arc $\cD_2$. 

To start with $E_1$, one can easily verify that the bound $g_{m,2}(z)\ll \frac{1}{\beta_k^2}e^{-\frac{3\pi^2}{2\beta_k}}$ obtained in Lemma \ref{errorterm} for $z\in \CC_1$ continues to hold for $z\in \cD_1$ and $m\geq 1$. Similarly, the proof of Lemma \ref{E1} mostly goes through, except that the length of $\cD_1$ is of order $\beta_k\, m^{-\frac{1}{3}}$. As a result we have now
$$
E_1\ll n^{-\frac{k}{4}+\frac{1}{2}}\,m^{-\frac{1}{3}}\,e^{\pi \sqrt{\frac{2kn }{3}}-\frac{3\pi}{2}\sqrt{\frac{6n}{k}}}.
$$
 
To establish the bound for $E_2$, we note that for $z\in \cD_1$ and $m\geq 1$, $g_{m,1}(z)\ll P(m,\beta)$. We can follow again roughly the proof of Lemma \ref{E2}, using now $g_{m,1}(z)\ll P(m,\beta)$ and that the length of $\cD_1$ is of order $\beta_k\, m^{-\frac{1}{3}}$. Then one obtains
$$
E_2\ll n^{-\frac{k}{4}-\frac{1}{2}}\,m^{-\frac{1}{3}}\,P(m,\beta_k)\,e^{\pi\sqrt{\frac{2kn}{3}}-4\pi\sqrt{\frac{6n}{k}}}.
$$
For $E_3$, we recall that Lemma \ref{mainterm} gives
$$
g_{m,1}(z)-\frac{z}{\beta_k}P(m,\beta_k)=O\!\left(\beta_k m^\frac{2}{3} P(m,\beta_k) \right).
$$
After substituting this bound in the definition of $E_3$ and again using the length of $\cD_1$, one obtains
$$   
E_3\ll n^{-\frac{k}{4}-1}\,m^\frac{1}{3}\,P(m,\beta_k) \,e^{\pi \sqrt{\frac{2kn}{3}}}.
$$
Finally, one can also verify that the proof for $E$ in Section \ref{sec:th1}  is now applicable with $M=m^{-\frac{1}{3}}$. This leads to
$$
E\ll n^{-\frac{k-6}{4}} e^{\pi \sqrt{\frac{kn}{6}}\left(1-\frac{3}{4}m^{-\frac{2}{3}}\right)}.
$$
Therefore the dominant pole is indeed the one for $z=O(n^{-\frac{1}{2}})$. Comparing now all error terms we see that $E_3$ is again the dominating error, which concludes the proof of Theorem \ref{th2}. 

\end{document}